\documentclass[a4paper,10pt]{amsart}
\usepackage{cancel} 
\usepackage{inputenc}
\usepackage{graphicx}
\usepackage{subcaption} 
\usepackage{amsmath,mathtools}
\usepackage{amsfonts}
\usepackage{amssymb}
\usepackage{mathrsfs}
\usepackage{color}
\usepackage{enumerate}
\usepackage[colorlinks]{hyperref}
\usepackage{algorithm}
\usepackage{mdframed}
\usepackage{algorithmic}
\usepackage{siunitx}
\usepackage{multirow}
\usepackage{stackrel}
\usepackage{booktabs}
\usepackage{tabularx}
\usepackage{arydshln}
\usepackage{cite}
\usepackage{url}
\usepackage{comment}
\usepackage{listings}

\usepackage{booktabs}

\usepackage[top=2.8cm, bottom=2.8cm, left=2.6cm, right=2.6cm]{geometry}

\newtheorem{theorem}{Theorem} 
\newtheorem{lemma}{Lemma}     
\newtheorem{definition}{Definition}
\newtheorem{Remark}{Remark}

\newtheorem{proposition}{Proposition}

\newcommand{\R}{\mathbb{R}}

\begin{document}
\title[Block Diagonally Preconditioned Multiple Saddle-Point Matrices]{Spectral Analysis of Block Diagonally Preconditioned Multiple Saddle-Point Matrices \\ with Inexact Schur Complements}

\author{L. Bergamaschi \and A. Mart\'{\i}nez \and M. Pilotto}
\date{}
\footnotetext[1]{Department of Civil Environmental and Architectural Engineering, University of Padua, Via Marzolo, 9, 35100 Padua, Italy,
\texttt{E-mail}: \url{luca.bergamaschi@unipd.it}}
\footnotetext[2]{Department of Mathematics, Informatics and Geosciences, University of Trieste, Via Weiss, 2, 34128 Trieste, Italy,
\texttt{E-mail}: \url{amartinez@units.it}}
\footnotetext[3]{International Centre for Numerical Methods in Engineering (CIMNE),
   Barcelona, Spain, \texttt{E-mail}: \url{marco.pilotto@upc.edu}}

\begin{abstract}
We derive eigenvalue bounds for symmetric block-tridiagonal multiple saddle-point systems preconditioned with block-diagonal Schur complement matrices. This analysis applies to an arbitrary number of blocks and accounts for the case where the Schur complements are approximated, generalizing the findings in \cite[Bergamaschi et al., Linear Algebra and its Applications, 2026]{bergamaschi2025eigenvalue}. Numerical experiments are carried out to validate the proposed estimates.
\end{abstract}

\maketitle
\section{Introduction}
We consider the iterative solution of a block tridiagonal multiple saddle-point linear system $\mathcal{A}x = b$, where 
\[
\mathcal{A}=\left[\begin{array}{ccccc}
    A_0 & B_1^T & 0 & \dots & 0\\
    B_1 & -A_1 & B_2^T & \ddots & \vdots\\
    0 & B_2 & A_2 & \ddots & 0\\
    \vdots & \ddots & \ddots & \ddots & B_N^T\\
    0 & \dots & 0 & B_N & (-1)^NA_N
\end{array}\right]
\]
We assume that $A_0\in\mathbb{R}^{n_0\times n_0}$ is symmetric positive definite, all other square block matrices $A_k\in\mathbb{R}^{n_k\times n_k}$ are symmetric positive semi-definite and $B_k\in\mathbb{R}^{n_k\times n_{k-1}}$ have full rank (for $k=1,\dots, N$). 
We assume also that $n_k\leq n_{k-1}$ for all $k$. These conditions are sufficient to ensure the invertibility of $\mathcal{A}$. In this work, we develop eigenvalue bounds for symmetric multiple saddle-point matrices $\mathcal{A}$, preconditioned with a block diagonal Schur complement preconditioner,
$\mathcal{P}_D$:
\[
    \mathcal{P}_D = \texttt{blkdiag}(S_0,\dots,S_N), \quad\quad\text{with }\ S_0=A_0, \quad \text{and }\ S_k=A_k +B_kS_{k-1}^{-1}B_k^T,\quad k=1,\dots,N.
\]
Since the application of $\mathcal P_D$ reveals impractical, especially for non non-trivial number of blocks $N$, we will consider
an inexact version of this preconditioner, in which some (or all) Schur complements are approximately computed and/or applied.

Linear systems involving matrix $\mathcal{A}$ arise mostly with $N =2$ (double saddle--point systems) in many scientific applications, including constrained quadratic programming \cite{Han}, magma--mantle dynamics \cite{Rhebergen}, 
liquid crystal director modeling \cite{RamGar2023} or in the coupled Stokes--Darcy problem \cite{greifhe2023, Szyld, BeikBenzi2022, greif2026}, and the preconditioning of such linear systems has been considered in, e.g., \cite{Balani-et-al-2023a, Balani-et-al-2023b, Szyld, Benzi2018, BeikBenzi2022}.
In particular, block diagonal preconditioners for  matrix {$\mathcal{A}$} have been considered in e.g.
\cite{Bradley,PPNLAA24,bergamaschi2025eigenvalue}.

More generally, multiple saddle-point linear systems with $N > 2$ have recently attracted the attention of a number of researchers. Such systems often arise from modeling multiphysics processes, i.e., the simultaneous simulation of different aspects of 
physical systems and the interactions among them. 
A new symmetric positive definite preconditioner,
proposed in \cite{pearson2023symmetric}, is successfully tried on constrained optimization problems up to $N=4$. 
The spectral properties of the resulting preconditioned matrix have been described in \cite{BMPP_COAP24} for $N=2$,
and later extended to multiple saddle point systems in \cite{BB2026}. 
Preconditioning of the $4 \times 4$ saddle-point linear system ($N=3$), although
with a different block structure, has been addressed in  \cite{BSZ2020} to solve a class of optimal control problems.
A practical preconditioning strategy for multiple saddle-point linear systems, based on sparse approximate inverses of the diagonal blocks of the block diagonal Schur complement preconditioning matrix, is proposed in \cite{FerFraJanCasTch19}, for the solution of
coupled poromechanical models and the mechanics of fractured media. Theoretical analysis of such preconditioners
has been carried on in \cite{BerFerMar25}.

This work is the natural development of the analysis in
 \cite{SZ} and, more recently, in  \cite{bergamaschi2025eigenvalue}, in which the exact $\mathcal P_D$ preconditioner
 is considered. In \cite{SZ} the eigenvalues of 
$\mathcal P_D^{-1} \mathcal A$ 
are exactly computed, in the simplified case where $A_k$ are $n_k \times n_k$
zero matrices for $k > 0$, and are related to the zeros of suitable Chebyshev-like polynomials. In 
\cite{bergamaschi2025eigenvalue} this constraint is removed, and very tight bounds on the eigenvalues 
of $\mathcal{P}_D^{-1} \mathcal A$ are developed in terms of the roots of a sequence of Chebyshev-like polynomials. 

In this work, we specifically address the case where the (inverse of the) Schur complements are applied approximately. The inexact
preconditioner $ {\mathcal{P}}$ is defined as
\begin{equation}
	\label{inexact}
	\mathcal{P}= \texttt{blkdiag}(\widehat S_0,\dots,\widehat S_N),  \qquad \text{with} \quad 
\begin{array}{lcl}
	&  & \widehat{S}_0 \approx A_0\\
	\tilde{S}_k &= A_k + B_k\widehat{S}^{-1}_{k-1}B_k^T& \widehat{S}_k\approx \tilde{S}_k.
\end{array} \end{equation}
Eigenvalue bounds for the preconditioned matrix ${\mathcal{P}} ^{-1} \mathcal A$ are established for an arbitrary number of blocks in terms of the extremal eigenvalues of the matrices
$ \widehat S_k^{-1} A_k$ and $\widehat S_k^{-1} \left(\tilde{S}_k - A_k\right).$

Arguably the most prominent Krylov subspace methods for solving $\mathcal A x = b$ are preconditioned variants of MINRES~\cite{minres} and GMRES~\cite{saad1986gmres}. In contrast to GMRES, the previously-discovered MINRES algorithm can explicitly exploit the symmetry of $\mathcal{A}$. As a consequence, MINRES features a three-term recurrence relation, which is beneficial for its implementation (low memory requirements because subspace bases need not be stored) and its purely eigenvalue-based convergence analysis (via the famous connection to orthogonal polynomials; see~\cite{fischer1996polynomial,greenbaum1997iterative}). Specifically, if the eigenvalues of the preconditioned matrix are contained within $[\rho^-_l, \rho^-_u] \cup [\rho^+_l, \rho^+_u]$, for $\rho^-_l < \rho^-_u < 0 < \rho^+_l < \rho^+_u$ such that $\rho^+_u - \rho^+_l = \rho^-_u - \rho^-_l$, then at iteration $k$ the Euclidean norm of the preconditioned residual $r_k$ satisfies the bound
\[
        \frac{\lVert r_k \rVert}{\lVert r_0 \rVert} \le 2 \left( \frac{\sqrt{\lvert \rho^-_l \rho^+_u \rvert} - \sqrt{\lvert \rho^-_u \rho^+_l \rvert}}{\sqrt{\lvert \rho^-_l \rho^+_u \rvert} + \sqrt{\lvert \rho^-_u \rho^+_l \rvert}} \right)^{\lfloor k / 2 \rfloor}.
\]
By contrast, GMRES needs to store subspace bases and its convergence analysis is in general dependent on the corresponding eigenspaces as well, which are more complicated to analyze than eigenvalues (see, e.g., \cite{embree2022descriptive}).

This outline of this paper is as follows. In Section \ref{sec2}, we introduce the block-diagonal Schur complement preconditioner. Then we establish a relation between the eigenvalues of the preconditioned linear system and a family of parametric polynomials. In Section \ref{sec3}, we derive bounds for their extremal zeros, which also provide bounds for the eigenvalues.
Section \ref{sec4} investigates how the zeros of the polynomials vary with respect to their parameters.
	In Section \ref{sec5}, we link the desired eigenvalue bounds with the extremal zeros of the polynomials for specific choices of the parameters. Section \ref{Sec6} presents numerical experiments on synthetic test cases that validate the theoretical results, focusing on the cases $N = 2,\ 3,\ 4$.
Section \ref{Sec7} is devoted to the application of the block diagonal preconditioner in the solution of the linear
system arising from the Mixed-Hybrid Finite Element discretization of the Biot model in three spatial dimensions.
In Section \ref{Sec8} we draw some conclusions.

\section{Eigenvalues of Preconditioned Saddle-Point Systems with inexact Schur complements}\label{sec2}
In practice, the use of exact Schur complements $S_i$ is prohibitive, especially for larger $i$, due to the recursive definition of the Schur complements. When an approximation for the Schur complements is adopted, the exact formulation in \cite{BMPP_COAP24} does not hold. Indeed, in the inexact case, we denote the approximated Schur complements as $\widehat{S}_k$. Each of these should approximate a perturbed Schur complement $\tilde{S}_i$ that takes into account the previous Schur complement approximations according to \eqref{inexact}.
Our preconditioned matrix reads in this case 
\[
    \begin{aligned}
    \mathcal{Q}_{in}&={\mathcal{P}}^{-1/2} \mathcal{A}{\mathcal{P}}^{-1/2}=\left[\begin{array}{ccccc}
        E_0 & R_1^T & 0 & & \\
        R_1 & -E_1 & R_2^T & & \\
        0 & R_2 & E_2 & \ddots& \\
         & & \ddots & \ddots & R_N^T\\
         & & & R_N& (-1)^NE^N\\
    \end{array}\right]\\
    \end{aligned}
\]
where
\[
    R_k = \widehat{S}_k^{-1/2}B_k\widehat{S}_{k-1}^{-1/2}, \ \ k=1,\dots, N; \quad E_k = \widehat{S}_k^{-1/2}A_k\widehat{S}_k^{1/2}, \ \ k = 0, \dots, N.
\]
The following identities hold:
\[
    R_k R_k^\top + E_k = \widehat{S}_k^{-1/2}\tilde{S}_k\widehat{S}_k^{-1/2}\equiv \overline{S}_k, \qquad 
    k = 1, \ldots, N.
\]
Componentwise, we write the eigenvalue problem $\mathcal{Q}u=\lambda u$, with $u=\left[u_1^T,\dots,u^T_{N+1}\right]^T$, as 

\begin{equation}
    \begin{array}{ccccccccc}
        (E_0-\lambda I)u_1 & + & R_1^Tu_2 & & & & & = & 0\\
        R_1u_1 & + &-(E_1+\lambda I) u_2 & +& R_2^Tu_3 & & & = & 0 \\
         & & R_2u_2 & + & (E_2-\lambda I)u_3 & + & R_3^T u_4 & = & 0 \\
         & & & \vdots & & \vdots & & \vdots &  \\
         & & R_{N-1}u_{N-1}& + & ((-1)^{N-1}E_{N-1}-\lambda I)u_{N} & + & R_N^Tu_{N+1} & = & 0\\
          & & & & R_N u_N & + &((-1)^{N}E_{N}-\lambda I)u_{N+1} & = & 0.
    \end{array}
    \label{eq1}
\end{equation}

The matrices $R_kR_k^T$ are all symmetric and positive definite. We define two indicators $\gamma_E^{(k)}$ and $\gamma_R^{(k)}$ using the Rayleigh quotient
	\begin{equation}
		\begin{array}{lll}
\alpha_E^{(i)} \equiv \lambda_{\min }\left(E_i\right), & \beta_E^{(i)} \equiv \lambda_{\max }\left(E_i\right), & \gamma_E^{(i)}\left(w\right)=\dfrac{w^{T} E_i w}{w^{T} w} \in\left[\alpha_{E}^{(i)}, \beta_{E}^{(i)}\right] \equiv \mathcal{I}_{E_i}, \quad i=0,\dots,N \\[.6em]
\alpha_R^{(i)} \equiv \lambda_{\min }\left(R_i R_i^{T}\right), & \beta_R^{(i)} \equiv \lambda_{\max }\left(R_i R_i^{T}\right), & \gamma_R^{(i)}\left(w\right)=\dfrac{w^{T} R_i R_i^{T} w}{w^{T} w} \in\left[\alpha_{R}^{(i)}, \beta_{R}^{(i)}\right] \equiv {\mathcal{I}}_{R_i}, \quad i=1,\dots,N.
\end{array}
		\label{gammaER}
	\end{equation} 
We define  the two vectors of parameters:
\[
\boldsymbol{\gamma}_E=\left[\gamma_E^{(0)},\dots,\gamma_E^{(N)}\right],\quad\quad \boldsymbol{\gamma}_R=\left[\gamma_R^{(1)},\dots,\gamma_R^{(N)}\right]
\]
and
introduce a new sequence of parametric polynomials depending on $\boldsymbol{\gamma}_E$ and $\boldsymbol{\gamma}_R$:
\begin{definition} 
\begin{equation} 
    \begin{aligned} U_0(\lambda,\boldsymbol{\gamma}_R,\boldsymbol{\gamma}_E) &= 1,\\ 
	    U_1(\lambda,\boldsymbol{\gamma}_R,\boldsymbol{\gamma}_E) &= \lambda -\gamma_E^{(0)},\\ 
    U_{k+1}(\lambda,\boldsymbol{\gamma}_R,\boldsymbol{\gamma}_E) &= (\lambda + (-1)^{k+1}\gamma_E^{(k)})U_k(\lambda,\boldsymbol{\gamma}_R,\boldsymbol{\gamma}_E)- \gamma_R^{(k)}U_{k-1}(\lambda,\boldsymbol{\gamma}_R,\boldsymbol{\gamma}_E), \qquad  k \ge 1 \end{aligned}
    \label{eq2} 
\end{equation} 
\end{definition}

\begin{proposition}\label{prop1}
    \textit{Since the polynomials of the sequence \eqref{eq2} are monic}
\[
    \lim_{\lambda\rightarrow+\infty}U_k(\lambda) = +\infty \quad\quad \lim_{\lambda\rightarrow-\infty}U_k(\lambda) = (-1)^k \cdot\infty.
\]
\end{proposition}

\noindent
\textcolor{black}{
With $\mathcal{I}_k$ we denote the union of the two, negative and positive, intervals containing the roots of the
polynomials $U_k(\lambda,\boldsymbol{\gamma}_E,\boldsymbol{\gamma}_R))$ over the admissible range of $\boldsymbol{\gamma}_E, \boldsymbol{\gamma}_R$. We now define a matrix-valued sequence:}

\begin{definition}
    \[
    \begin{aligned}
        Y_1(\lambda) &= \lambda I - E_0\\
	    Y_2(\lambda) &= R_1Y_1(\lambda)^{-1}R_1^T-\lambda I - E_1, \quad \lambda \not \in \mathcal I_{E_0}\\
        Y_{k+1}(\lambda) &= R_k Y_{k}(\lambda)^{-1}R_k^T + (-1)^{k}\lambda I - E_k, \qquad k \ge 1, ~ \emph{ and } ~ \lambda ~ \emph{ s.t. } ~ 0 \not \in \sigma\left(Y_{k}(\lambda)\right).
    \end{aligned}
    \]
\end{definition}
\noindent
We mention a technical lemma, generalizing \cite[Lemma 1]{BergaNLAA10} whose proof can be found in \cite{BMPP_COAP24}.
\begin{lemma}\label{lemma1}
    \textit{Let Y be a symmetric matrix valued function defined in $F\subset\mathbb{R}$ and}
    \[
        0\notin[\min\{\sigma(Y(\zeta))\},\max\{\sigma(Y(\zeta))\}]\quad \quad \textit{for all }\ \zeta\in F.
    \]
    \textit{Then, for arbitrary $s\neq0$, there exists a vector $v\neq 0$ such that}
    \[
       \frac{s^TY(\zeta)^{-1}s}{s^Ts}= \frac{1}{\gamma_Z} \quad\quad \textit{with }\ \gamma_Z = \frac{v^TY(\zeta)v}{v^Tv}.
    \]
\end{lemma}
\noindent
The next lemma, which links the sequence of polynomials with the sequence $\{Y_k(\lambda)\}$,  will be used in the proof of the subsequent Theorem \ref{theorem1}.

\begin{lemma}\label{lemma2}
    \textit{For every $u\neq0$, there is a choice of $\boldsymbol{\gamma}$ for which}
    \[
    \frac{u^TY_{k+1}(\lambda)u}{u^Tu}=(-1)^{k}\frac{U_{k+1}(\lambda)}{U_k(\lambda)}\quad\quad \textit{for all }\ \lambda\notin\bigcup_{j=1}^k\mathcal{I}_j.
    \]
\end{lemma}
\begin{proof}
This is shown by induction. For $k=0$ we have $\displaystyle\frac{u^TY_1(\lambda)u}{u^Tu}=\lambda -\gamma_E^{(0)}=\frac{U_1(\lambda)}{U_0(\lambda)}$ for all $\lambda\in\mathbb{R}$. If $k\geq1$, the condition $\lambda\notin\mathcal{I}_k$ , together with the inductive hypothesis $\displaystyle\frac{u^TY_{k}(\lambda)u}{u^Tu}=\frac{U_{k}(\lambda)}{(-1)^{k-1}U_{k-1}(\lambda)}$, implies invertibility of $Y_k(\lambda)$. Moreover, this is equivalent to the condition $0\notin[\min\{\sigma(Y(\zeta))\},\max\{\sigma(Y(\zeta))\}]$ that guarantees the applicability of Lemma \ref{lemma1}. Therefore, we can write 
\begin{equation}\label{eq3}
\begin{aligned}
    \frac{u^TY_{k+1}(\lambda)u}{u^Tu}&=\frac{u^TR_kY_{k}(\lambda)^{-1}R_k^Tu}{u^Tu}\gamma_R^{(k)} + (-1)^k\lambda-\gamma_E^{(k)}\\
    &\underset{w=R^T_k u}{=}\frac{w^TY_{k}(\lambda)^{-1}w}{w^Tw}\gamma_R^{(k)} + (-1)^k\lambda-\gamma_E^{(k)}.
\end{aligned}
\end{equation}
We then apply Lemma \ref{lemma1} and the inductive hypothesis to write 
\[
	\frac{w^TY_{k}(\lambda)^{-1}w}{w^Tw}=(-1)^{k-1} \frac{U_{k-1}(\lambda)}{U_{k}(\lambda)}.
\]
Substituting into \eqref{eq3} and using relation \eqref{eq2} we get
        \begin{eqnarray*}\frac{w^TY_{k}(\lambda)^{-1}w}{w^Tw}\gamma_R^{(k)} + (-1)^{k}\lambda -\gamma_E^{(k)} &= &
		(-1)^{k-1}\gamma_R^{(k)}\frac{U_{k-1}(\lambda)}{U_{k}(\lambda)} + (-1)^{k}\lambda -\gamma_E^{(k)}  \\
		&=& (-1)^k \frac{-U_{k-1} \gamma_R^{(k)} (\lambda -1) + (\lambda +(-1)^{k+1} \gamma_E^{(k)})U_k(\lambda)}{U_{k-1}(\lambda)}   \\
                &=& (-1)^{k} \frac{U_{k+1}(\lambda)}{U_k(\lambda)}.
        \end{eqnarray*}

\end{proof}

\begin{theorem}\label{theorem1}
    \textit{The eigenvalues of ${\mathcal{P}}^{-1}\mathcal{A}$ are located in $\displaystyle\bigcup_k^{N}\mathcal{I}_{k+1}$ }
\end{theorem}

\begin{proof}
	 To show the statement, we define a candidate {eigenpair $(\lambda, \left[u_1^T,\dots,u_{k+1}^T\right]^T)$}, and prove by induction that for every $k \le N+1$ either
        \[      \text{(i)} \ \lambda \in  \mathcal{I}_k \quad \text{or} \quad  \text{(ii)}  \ u_k = Y_k^{-1} R_k^T u_{k+1}, \]
        and for $k = N+1$ only condition (i) can hold.
Assume that $\lambda\notin\mathcal{I}_{E_0}$ (for $k=0$), then $Y_1(\lambda)$ is invertible. From the first equation of \eqref{eq1} we obtain 
\begin{equation}
(\lambda I -E_0)u_1 = R_1^T u_2\quad\rightarrow\quad Y_1(\lambda)u_1 = R_1^T u_2;
\label{eq4}
\end{equation}
inserting \eqref{eq4} into the second row of \eqref{eq1} yields 
\begin{equation}
(R_1Y_1(\lambda)^{-1}R_1^T-\lambda I -E_1)u_2=R_2^Tu_3 \quad\rightarrow\quad Y_2(\lambda)u_2=R_2^Tu_3.
\label{eq5}
\end{equation}
Pre-multiplying the left hand side of \eqref{eq5} by $\frac{u_2^T}{u_2^Tu_2}$ we get
\begin{equation}
\begin{aligned}
    \frac{u_2^TY_2(\lambda)u_2}{u_2^Tu_2} &=\frac{u_2^T(R_1Y_1(\lambda)^{-1}R_1^T)u_2}{u_2^Tu_2}-\lambda -\frac{u_2^TE_1u_2}{u_2^Tu_2}\\
    &\underset{s=R_1^T u_2}{=}\frac{s^TY_1(\lambda)^{-1}s}{s^Ts}\gamma_R^{(1)}-\lambda - \gamma_E^{(1)}.
\end{aligned}
\end{equation}
By Lemma \ref{lemma2}, this expression, considering \eqref{eq2} equals 
\[
    \frac{U_0(\lambda)}{U_1(\lambda)}\gamma_R^{(1)} - \lambda -\gamma_E^{(1)}= \frac{U_2(\lambda)}{-U_1(\lambda)}.
\]
If $u_3=0$ then $\lambda$ being an eigenvalue implies that $U_2(\lambda,\boldsymbol{\gamma}_E, \boldsymbol{\gamma}_R)=0$. Otherwise, if $\lambda\notin\mathcal{I}_2$ then $Y_2(\lambda)$ is invertible.
Assume now the inductive hypothesis holds for $k-1$. If $\lambda\notin\mathcal{I}_{k-1}$, then $Y_{k-1}(\lambda)$ is definite and invertible. We can write 
\[
(R_{k-1}Y_{k-1}(\lambda)R_{k-1}^T+(-1)^{k-1}\lambda-\gamma_E^{(k-1)})u_k=R_k^Tu_{k+1}\quad\rightarrow\quad u_k=Y_{k}(\lambda)^{-1}u_{k+1}.
\]
Now, if $\lambda$ is an eigenvalue with $u_{k+1}=0$, from
\[
0=\frac{u_k^TY_{k}(\lambda)u_k}{u_k^Tu_k}=\frac{U_k(\lambda)}{(-1)^{k-1}U_{k-1}(\lambda)},
\]
we have $U_{k}(\lambda,\boldsymbol{\gamma}_E, \boldsymbol{\gamma}_R)=0$ and hence $\lambda\in\mathcal{I}_k$. 
	Otherwise for any $\lambda\notin\mathcal{I}_k$ we may write
\[
u_k=Y_{k}(\lambda)^{-1} R_k^Tu_{k+1}.
\]
The induction process ends for $k=N+1$. In this case, we have that 
\[
Y_{N+1}u_{N+1}=0,
\]
and the condition $\lambda\in\mathcal{I}_{N+1}$ must be verified, noticing that $u_{N+1}=0$ would imply that also $u_N=\dots=u_1=0$ contradicting the definition of an eigenvector.
\end{proof}

\section{Bounds for the extremal zeros of the sequence \eqref{eq2} }
\label{sec3}
In the previous section, we have seen that there is a strict relation between the eigenvalues of the preconditioned matrix and
the zeros of the sequence of polynomials \eqref{eq2}.
In this section, we will characterize such zeros. We first prove that, for each admissible combination of the $\boldsymbol{\gamma}-$ parameters,  all polynomials in the sequence have real and distinct zeros.
Furthermore, zeros of consecutive polynomials display an interlacing property.

\begin{lemma}\label{lemma3}
    \textit{Let ${\gamma}^{(k)}_E\in\left[\alpha_E^{(k)}, \beta_E^{(k)}\right]$ and ${\gamma}^{(k)}_R\in\left[\alpha_R^{(k)}, \beta_R^{(k)}\right]$ for $k=1,\dots,N$. Then $U_{k+1}(\lambda,\boldsymbol{\gamma}_E, \boldsymbol{\gamma}_R)$ has $k+1$ real and distinct roots.}
\end{lemma}

\begin{proof}
The following proof is based on the proof of Lemma 4 in \cite{bergamaschi2025eigenvalue}. Define, for $k>0$
	\begin{equation}
		\label{ak}
a_{k+1}(\lambda) = U_k(\lambda)U_{k+1}^\prime(\lambda)-U^\prime_k(\lambda)U_{k+1}(\lambda).
	\end{equation}
The polynomials of the sequence $U_k$ has degree $\lambda^{k+1}$. 
We observe that the highest order monomial in $a_{k+1}$ is $(k+1)\lambda^k\cdot\lambda^k-k\lambda^{k-1}\lambda^{k+1} = \lambda^{2k}$, which shows that $a_{k+1}$ is a monic polynomial. We set, for brevity, $c_{k+1} = (-1)^{k+1}\gamma_E^{(k)}$. From \eqref{eq2} we can write
\[
\begin{aligned}
    U_{k+1}(\lambda)&=(\lambda + c_{k+1})U_k(\lambda)- \gamma_R^{(k)}U_{k-1}(\lambda)\\
    U^\prime_{k+1}(\lambda)&= U_k(\lambda) + (\lambda + c_{k+1})U_k^\prime(\lambda)-\gamma_R^{(k)}U^\prime_{k-1}(\lambda).
\end{aligned}
\]
Now, multiplying the first equation by $U_k^\prime(\lambda)$ and the second one by $U_k(\lambda)$ yields 
\[
\begin{aligned}
    U_{k+1}(\lambda)U_k^\prime(\lambda)&=(\lambda + c_{k+1})U_k(\lambda)U_k^\prime(\lambda)- \gamma_R^{(k)}U_k^\prime(\lambda)U_{k-1}(\lambda)\\
    U_k(\lambda)U^\prime_{k+1}(\lambda)&= U_k(\lambda)^2 + (\lambda + c_{k+1})U_k(\lambda)U_k^\prime(\lambda)-\gamma_R^{(k)}U_k(\lambda)U^\prime_{k-1}(\lambda).
\end{aligned}
\]
The difference between the two expressions results in 
\[
a_{k+1}(\lambda) = U_k(\lambda)U_{k+1}^\prime(\lambda)-U^\prime_k(\lambda)U_{k+1}(\lambda) = U_k(\lambda)^2 + \gamma_R^{(k)}a_k(\lambda).
\]
Since $a_1(\lambda) = U_0U_1^\prime - U_0^\prime U_1 = 1$, by induction the sequence of polynomials $\{a_{k+1}(\lambda)\}$ is positive for all $\lambda\in\mathbb{R}$. Recall that ${\gamma}^{(k)}_R >0$ by assumption.
We now proceed by induction to prove the statement, which we assume to be true for index k, namely that there are real numbers 
\[
\xi_1^{(k)}<\xi_2^{(k)}<\dots<\xi_k^{(k)}
\]
such that $U_k(\xi_j^{(k)})=0$, $j=1,\dots,k$. Since $U_k(\lambda)$ is a monic polynomial it turns out that $(-1)^iU^\prime_k(\xi_{k-i}^{(k)})>0$, for $i=0,\dots,k-1$. For every $j$ it holds 
\begin{equation} 
0<a_{k+1}(\xi_j^{(k)})={U_k(\xi_j^{(k)})}U_{k+1}^\prime(\xi_j^{(k)})-U^\prime_k(\xi_j^{(k)})U_{k+1}(\xi_j^{(k)})=-U^\prime_k(\xi_j^{(k)})U_{k+1}(\xi_j^{(k)}),
\label{aknew}
\end{equation}
\textcolor{black}{
which shows that the sign of $U_{k+1}(\xi_j^{(k)})$ is always opposite to that of $U^\prime_k(\xi_j^{(k)})$. Moreover, since
$\text{sgn}(U^\prime_k(\xi_j^{(k)})) = - \text{sgn}(U^\prime_k(\xi_{j+1}^{(k)}))$ we conclude
that $U_{k+1}(\xi_j)$ and $U_{k+1}(\xi_{j+1})$ have opposite signs.}
This gives the alternating sign of the values of the polynomial $U_{k+1}(\lambda)$ at the roots of $U_k(\lambda)$. Overall, we have $k+1$ intervals in which  Bolzano's theorem can be applied, yielding $k+1$ real roots.
To conclude, the strict positivity of $a_k(\lambda)$ ensures that $U_{k+1}(\lambda)$ cannot have multiple roots. Indeed, consider $\lambda_\star$ such that $U_{k+1}(\lambda_\star) =U^\prime_{k+1}(\lambda_\star)=0$ it would follow that $a_{k+1}(\lambda_\star)=0$, which is impossible since $a_{k+1}(\lambda)>0, \ \forall\ \lambda\in \mathbb{R}$.
\end{proof}

The next two propositions will give some more detail on the sign of the roots of $U_k$.

\begin{proposition}
    Given the sequence of polynomials \eqref{eq2}, for all $k\geq1$
    \[
    \textnormal{sgn}(U_{2k-1}(0,\boldsymbol{\gamma}_E,\boldsymbol{\gamma}_R))=  \textnormal{sgn}(U_{2k}(0,\boldsymbol{\gamma}_E,\boldsymbol{\gamma}_R)) = (-1)^k
    \]
    In other words, the signs of $U_k(0)$ follow the repeating pattern $--,++,--,++,\dots$ starting from $k=1$.
    \label{prop2}
\end{proposition}
\begin{proof}
    We prove this by induction on $k$. The claim holds for $k=1$
    \[
    U_1(0) = -\gamma^{(0)}_E<0, \quad\quad U_2(0)=-\gamma_E^{(1)}\gamma_E^{(0)}-\gamma_R^{(1)}<0.
    \]
    Assume that the claim holds for some $k\geq1$, that is the sign of $U_{2k-1}(0)$ and $U_{2k}(0)$ is $s=(-1)^k$
    \begin{eqnarray*}
    \textnormal{sgn}(U_{2k+1}(0))&=&(-1)^{2k+1}\textnormal{sgn}(U_{2k}(0))-\textnormal{sgn}(U_{2k-1}(0)) = -s = (-1)^{k+1}.
    \\
    \textnormal{sgn}(U_{2k+2}(0))&=&(-1)^{2k+2}\textnormal{sgn}(U_{2k+1}(0))-\textnormal{sgn}(U_{2k}(0)) = -s = (-1)^{k+1}.
    \end{eqnarray*}
\end{proof}
\begin{proposition}
	Given the sequence of polynomials \eqref{eq2}, if $k\in \mathbb{N}$ is even then $U_k(\lambda)$ has $k/2$ positive and $k/2$ negative roots 
	$\{\xi_j^{(k)}\}_{j = 1, \ldots, k}$, while $U_{k+1}(\lambda)$ has $k/2+1$ positive and $k/2$ negative roots 
	$\{\xi_j^{(k+1)}\}_{j = 1, \ldots, k+1}$.
     \label{prop3}
\end{proposition}
\begin{proof}
    Consider $k=1$. We know that 
    \[
    U_1(0)<0,\quad\quad \lim_{\lambda\rightarrow+\infty}U_1(\lambda)=+\infty,\quad\quad \lim_{\lambda\rightarrow-\infty}U_1(\lambda)=-\infty 
    \]
    Therefore, the unique root of $U_1(\lambda)$ must be positive. For $k=2$, the interlacing property implies that $U_2(\lambda)$ has
	a root in $[-\infty, \xi^{(1)})$ and another one in $[\xi^{(1)}, +\infty)$
    \[
    U_2(0)<0,\quad\quad \lim_{\lambda\rightarrow+\infty}U_2(\lambda)=+\infty,\quad\quad \lim_{\lambda\rightarrow-\infty}U_2(\lambda)=+\infty 
    \]
    thus the polynomial has one negative and one positive root.\\
    Assume that the statement holds for some \textit{even} index $k$. The the roots of $U_k(\lambda)$ are
    \[
    \xi_1^{(k)}<\xi_2^{(k)}<\dots<\xi_m^{(k)}<0<\xi_{m+1}^{(k)}<\dots<\xi_{2m}^{(k)}\quad \text{ with }\: 2m=k
    \]
    \textcolor{black}{
    Now consider $U_{k+1}$, by the interlacing property we have
    \[ \xi_1^{(k+1)}< \xi_1^{k} < \ldots < \xi_m^{(k+1)} < \xi_m^{(k)}<0,\]
which shows that at least $m$ roots are negative. Again for the interlacing property, it must hold that $\xi_{m+1} ^{(k+1)} \in (\xi_m^{(k)},\xi_{m+1}^{(k)})$.
We have to prove that $\xi_{m+1} ^{(k+1)} > 0$. We distinguish two cases:
\begin{itemize}
\item \fbox{$U_k(0) > 0$}. Then it must be $U_k'(\xi_{m+1}^{(k)}) < 0$, which implies by \eqref{aknew} that $U_{k+1}(\xi_{m+1}^{(k)}) > 0$. Since $k$ is even, by Proposition \ref{prop3}, $\text{sgn}(U_{k+1}(0)) = - \text{sgn}(U_k(0)) =  -1$, and, therefore, $U_{k+1}(\xi_{m+1}^{(k)}) > 0$ implies  that $\xi_{m+1} ^{(k+1)} \in (0,\xi_{m+1}^{(k)})$, as desired.
\item \fbox{$U_k(0) < 0$}. Then it must be $U_k'(\xi_{m+1}^{(k)}) > 0$, which implies by \eqref{aknew} that $U_{k+1}(\xi_{m+1}^{(k)}) < 0$. Since $k$ is even, by Proposition \ref{prop3}, $\text{sgn}(U_{k+1}(0)) = - \text{sgn}(U_k(0)) =  +1$, and, therefore, $U_{k+1}(\xi_{m+1}^{(k)}) < 0$ implies, again,  $\xi_{m+1} ^{(k+1)} \in (0,\xi_{m+1}^{(k)})$, as desired.
\end{itemize}
It remains to prove that $U_{k+2}$ has $m+1$ negative and $m+1$ positive zeros, specifically, that 
$\xi_{m+1}^{(k+2)} < 0.$ Reasoning as before:
\begin{itemize}
\item \fbox{$U_{k+1}(0) > 0$}. Then it must be $U_{k+1}'(\xi_{m+1}^{(k+1)}) > 0 \ \Longrightarrow \ U_{k+2}(\xi_{m+1}^{(k)}) < 0$. Since $k+1$ is odd, by Proposition \ref{prop3}, $\text{sgn}(U_{k+2}(0)) = \text{sgn}(U_{k+1}(0)) =  +1$, and, therefore, $U_{k+2}(\xi_{m+1}^{(k+1)}) < 0$ implies  that $\xi_{m+1} ^{(k+2)} \in (\xi_{m+1}^{(k+1)}, 0)$, as desired.
\item \fbox{$U_{k+1}(0) < 0$}. Then it must be $U_{k+1}'(\xi_{m+1}^{(k+1)}) < 0 \ \Longrightarrow \ U_{k+2}(\xi_{m+1}^{(k)}) > 0$. Since $\text{sgn}(U_{k+2}(0)) = \text{sgn}(U_{k+1}(0)) =  -1$, then $U_{k+2}(\xi_{m+1}^{(k+1)}) > 0$ implies  that $\xi_{m+1} ^{(k+2)} \in (\xi_{m+1}^{(k+1)}, 0)$, as desired.
\end{itemize}
}
\end{proof}

Now we are ready to relate the extremal zeros of the polynomials $U_k$ to the eigenvalues of the preconditioned matrix, which, in view of Theorem \ref{theorem1} are located in $\displaystyle \bigcup_{k=0}^{N}\mathcal{I}_{k+1}$.
The next result will characterize more precisely this set.
\begin{theorem}
		\begin{equation}
			\label{N+1odd}
\bigcup_{k=0}^{N}\mathcal{I}_{k+1} = 			 [\xi^{(N+1)}_{-,LB},\ b_{N+1}] \cup [a_{N+1}, \ \xi^{(N+1)}_{+,UB}],
		\end{equation}
	where	
	\begin{eqnarray}
		\label{ak} b_{N+1} &=&  \max\{\xi_{-,UB}^{(k)}, \ 0 \le k \le N, \ k \ \textnormal{even}\} \\
		\label{bk} a_{N+1} &=&  \min\{\xi_{+,LB}^{(k)}, \ 0 \le k \le N, \ k \ \textnormal{odd}\}.
	\end{eqnarray}
    \label{theorem2}
\end{theorem}
\begin{proof}
    By Theorem \ref{theorem1}, the eigenvalues of $\mathcal{P}^{-1}\mathcal{A}$ are contained in the union of the intervals $\mathcal{I}_{k}$, which are defined by the roots of the polynomials of the sequence \eqref{eq2}.  
    From Lemma \ref{lemma3}, the roots of $U_{k+1}$ and $U_k$ strictly interlace
    \[
    \xi_1^{(k+1)} <\xi_1^{(k)}<\xi_2^{(k-1)}<\cdots<\xi_k^{(k-1)}<\xi_k^{(k)}<\xi_{k+1}^{(k+1)}
    \]
    Consequently
    \[
   \xi_{-,LB}^{(N+1)} < \ldots < \xi_{-,LB}^{(k+1)}<\xi_{-,LB}^{(k)}\quad\text{ and }\quad \xi_{+,UB}^{(N+1)} > \ldots > \xi_{+,UB}^{(k+1)}>\xi_{+,UB}^{(k)} 
    \]
    which shows that the extremal negative and positive root bounds increase monotonically with $k$. 
\textcolor{black}{
    It remains to derive the  negative upper bound and the positive lower bound, taking into account the the proof of Proposition \ref{prop3}.
    \begin{itemize}
        \item 
    Consider the lower bound of the positive eigenvalues. If $k$ is odd, then $\xi_{+,LB}^{(k+1)} > \xi_{+,LB}^{(k)}$, which implies that
    \[ \min_{1\le k \le N+1} \xi_{+,LB}^{(k)}  = \min_{1 \le k \le N+1, \ k \ \text{odd}}  \xi_{+,LB}^{(k)} .\]
        \item 
    Consider the upper bound of the negative eigenvalues. If $k$ is even, then $\xi_{-,UB}^{(k+1)} < \xi_{-,UB}^{(k)}$, which implies that
    \[ \max_{1 \le k \le N+1} \xi_{-,UB}^{(k)}  = \max_{1 \le k \le N+1, \ k \ \text{even}}  \xi_{+,LB}^{(k)} .\]
    \end{itemize}}
	Summarizing this result, we obtain the upper bounds in \eqref{ak} and \eqref{bk}.
\end{proof}

Now the question 
	is: for which values of the parameters $\gamma^{(i)}_E, \gamma^{(i)}_R $ the extremal values of the roots is attained?
	  To show how the roots of the polynomials of the sequence $U_k(\lambda,\boldsymbol{\gamma})$ move with the parameters, we consider a vector of parameters $\boldsymbol \gamma \in \R^d$, and a polynomial
        $q(\lambda (\boldsymbol \gamma), \boldsymbol \gamma)$. At
        a root $\lambda(\boldsymbol \gamma^*)  \equiv \lambda^*$, $q(\lambda(\boldsymbol{\gamma}),\boldsymbol{\gamma})$ satisfies
    \begin{equation}
        q(\lambda(\boldsymbol{\gamma}),\boldsymbol{\gamma})=0.
    \end{equation}
    The implicit function theorem delivers the existence of a neighborhood $V \in \R^d$ of $\boldsymbol{\gamma}^*$ and a continuously
        differentiable function $\lambda:V\rightarrow\mathbb{R}$ such that
    \[
    \frac{\partial q(\lambda(\boldsymbol{\gamma}))}{\partial\gamma_j}=\frac{\partial q}{\partial\lambda}\cdot\frac{\partial\lambda(\boldsymbol{\gamma})}{\partial\gamma_j}
    \]
    The interlacing property just proved ensures that all roots of polynomials $U_k$ in the sequence are simple, for any combination of
    the parameters and therefore
	$\dfrac{\partial U_k}{\partial\lambda}(\lambda(\boldsymbol{\gamma^*}),\boldsymbol{\gamma^*}) \ne 0$, we have
    \begin{equation}
        \frac{\partial\lambda}{\partial\gamma_j}(\boldsymbol{\gamma^*})=-\frac{\partial U_k}{\partial\gamma_j}(\boldsymbol{\gamma^*})\biggl/\frac{\partial U_k}{\partial\lambda}(\lambda(\boldsymbol{\gamma^*}),\boldsymbol{\gamma^*}) \quad \ \text{for all}\ j=1,\dots,d.
        \label{eq10}
    \end{equation}
    To characterize the monotonicity of a root with respect to the parameters, we will use the previous result to show that
    \begin{equation}
	    \textnormal{sgn}\left(\frac{\partial\lambda}{\partial\gamma_j}(\boldsymbol{\gamma^*})\right)=
	    -\textnormal{sgn}\left(\frac{\partial U_k}{\partial\gamma_j}(\boldsymbol{\gamma^*})\right) 
	    \textnormal{sgn}\left(\frac{\partial U_k}{\partial\lambda}(\lambda(\boldsymbol{\gamma^*}),\boldsymbol{\gamma^*}) \right).
        \label{eq10b}
    \end{equation}
    The following Lemma will allow us to state that the extremal values of the zeros of $U_{k+1}(\lambda,\boldsymbol{\gamma}_E,\boldsymbol{\gamma}_R)$ are obtained at the extremal values of $\boldsymbol{\gamma}_E$ and $\boldsymbol{\gamma}_R$, namely $\{\alpha_E^{(k)},  \beta_E^{(k)}; \ \alpha_R^{(k)}, \beta_R^{(k)}\}$.
\begin{lemma}
	Let $\xi$ be a zero of $U_{k+1}$ and $\gamma^{(j)}_E \in [\alpha_E^{(j)}, \beta_R^{(j)}]$,
	and $\gamma^{(j)}_R \in [\alpha_R^{(j)}, \beta_R^{(j)}]$,
	(for $j \le k $).
                Then either
		$\dfrac {\partial U_{k+1}}{\partial \gamma^{(j)}_E}(\xi)  \ne 0 $ (respectively, 
		$\dfrac {\partial U_{k+1}}{\partial \gamma^{(j)}_R}(\xi)  \ne 0 $)
		or $\xi$ is independent of $\gamma_E^{(j)}$ (respectively, $\gamma_R^{(j)}$).
\end{lemma}
        \begin{proof}
                Since the dependence of $U_{k+1}$ on $\gamma_*^{(j)}$, $j \le k$ is linear, there are suitable polynomials $r_1(x, \boldsymbol \gamma)$, $r_2(x, \boldsymbol \gamma)$ independent of $\gamma_*^{(j)}$,
                such that
                \[ U_{k+1} (x, \boldsymbol \gamma) = \gamma_*^{(j)} r_1(x,\boldsymbol \gamma)  + r_2(x,\boldsymbol \gamma).\]
                Then,
                \[ \frac{\partial U_{k+1}}{\partial \gamma_j} (\xi, \boldsymbol \gamma) = r_1(\xi, \boldsymbol \gamma).\]
                If $r_1(\xi, \boldsymbol \gamma) = 0$, then $0 = U_{k+1} (\xi, \boldsymbol \gamma) = r_2(\xi, \boldsymbol \gamma)$.
                \textcolor{black}{In this case $\xi$ is simultaneously a root of $r_1(\xi, \boldsymbol \gamma)$ and
                $r_2(\xi, \boldsymbol \gamma)$, which do not depend on $\gamma_*{(j)}$, meaning
                that the value of the root is independent of $\gamma_*^{(j)}$.}
        \end{proof}

At this stage, we have a first answer to the problem of bounding the eigenvalues of ${\mathcal P}^{-1} \mathcal A$.
We simply evaluate the extremal zeros of the polynomials described by Theorem \ref {theorem2} for each combination 
of $\gamma_E^{(k)} \in \{\alpha_E^{(k)},\ \beta_E^{(k)}\},\ \gamma_R^{(k)} \in  \{\alpha_R^{(k)},\ \beta_R^{(k)}\}, k = 1, \ldots N$. 

In the next section, we will give some further insight into the monotonicity of a given zero $\xi$ in terms of
the $\gamma$ parameters. This will allow us to reduce the complexity of the algorithm.
\section{How zeros of $\{U_k\}$ move depending on $\gamma_E^{(j)}, \gamma_R^{(j)}$}\label{sec4}
Consider the sequence of polynomials \eqref{eq2} and 
fix an index $0\leq m\leq k$. For each $k\geq0$ we can characterize 
the derivatives of $U_{k+1}(\xi)$ with respect to ${\gamma}_E^{(m)}, {\gamma}_R^{(m)}$, when $\xi$ is a zero of $U_{k+1}$.

\begin{lemma}
Let $\xi$ be a root of $U_{k+1}(\lambda)$. Then, for $m \le k$,
\begin{align}
	\frac{\partial U_{k+1}}{\partial\gamma_E^{(m)}}(\xi)&=(-1)^{m+1} \frac{\left(U_m(\xi)\right)^2}{U_k(\xi)}\prod_{i = m+1}^k \gamma_R^{(i)},\label{eq13} \\
	\frac{\partial U_{k+1}}{\partial\gamma_R^{(m)}}(\xi)&=- \frac{U_m(\xi) U_{m-1}(\xi)}{U_k(\xi)}\prod_{i = m+1}^k \gamma_R^{(i)}.
\label{eq14}
\end{align}
\label{lemma5}
\end{lemma}

\begin{proof}
If $m=k$ we compute directly from \eqref{eq2},  
\begin{equation}
	\label{starting}
  \frac{\partial U_{k+1}}{\partial\gamma_E^{(k)}}(\lambda) = (-1)^{k+1}U_k(\lambda), \qquad
	\frac{\partial U_{k+1}}{\partial\gamma_R^{(k)}}(\lambda) = - U_{k-1}(\lambda),
\end{equation}
which correspond to \eqref{eq13} and \eqref{eq14}, if $m = k$.
Let now $m < k$,
We denote for brevity with $U'_{k,m}$ both $\dfrac{\partial U_{k+1}(\xi)}{\partial\gamma_E^{(m)}}$ and
$\dfrac{\partial U_{k+1}(\xi)}{\partial\gamma_R^{(m)}}$, since the two sequences obey (for $k > m)$
the same recurrence relation:
\begin{equation}
	\label{recder}
	U'_{k+1,m}(\xi) = \underbrace{\left(\xi + (-1)^{k+1} \gamma_E^{(k)} \right)}_{\eta_k}  U'_{k,m}(\xi)-\gamma_R^{(k)}U'_{k-1,m}(\xi)
	\equiv \eta_k U'_{k,m}(\xi)-\gamma_R^{(k)}U'_{k-1,m}(\xi),
\end{equation}
differing only for the starting point \eqref{starting} obtained  with $k=m$. \\
\noindent
\textcolor{black}{
With the previous definition of $\eta_k$ we can rewrite the recurrence \eqref{eq2} at a root $\xi$ of $U_{k+1}$ as
\[ 0 = U_{k+1} (\xi) = \eta_k U_{k}(\xi) - \gamma_R^{(k)}U_{k-1}(\xi) \quad \Longrightarrow  \quad \eta_k = \gamma_R^{(k)}\frac{U_{k-1}(\xi)}{U_k(\xi)} 
\]}
hence \eqref{recder} becomes 
\begin{equation}
	U'_{k+1,m}(\xi) = \frac{\gamma_R^{(k)}}{U_k(\xi)}\underbrace{\left[U_{k-1}(\xi)U'_{k,m}(\xi)-U_k(\xi)U'_{k-1,m}(\xi)\right]}
	_{W_k(\xi)}
	\equiv \frac{\gamma_R^{(k)}}{U_k(\xi)} W_k(\xi).
    \label{eqlemma3}
\end{equation}
We now re-write $W_{s+1}(\xi)$ in terms of $W_s(\xi)$, for $m+1 \le s \le k$.
\[
\begin{aligned}
	W_{s+1}(\xi)&=U_{s}(\xi)U'_{s+1,m}(\xi)-U_{s+1}(\xi) U_{s,m}(\xi)  \\
&=U_s(\xi)
	\left(\eta_sU'_{s,m}(\xi)-\gamma_R^{(s)}U'_{s-1,m}(\xi)\right)-U_{s+1}(\xi)U'_{s,m}(\xi)\\
	&=U'_{s,m}(\xi)\left(\eta_s U_s(\xi) - U_{s+1}(\xi)\right)-\gamma_R^{(s)}U_s(\xi)U'_{s-1,m}(\xi)\\
	&=\gamma_R^{(s)}\left(U_{s-1}U'_{s,m}(\xi)-U_s(\xi)U'_{s-1,m}(\xi)\right) = \gamma_R^{(s)}W_s(\xi).
\end{aligned}
\]
By induction, we have that 
\[
	W_{k}(\xi) = \gamma_R^{(k-1)} \gamma_R^{(k-2)} \ldots \gamma_R^{(m+1)}  W_{m+1}(\xi).
\]
\textcolor{black}{
To evaluate $W_k(\xi)$, we can now compute $W_{m+1}(\xi)$, distinguishing now the cases in which the derivative is obtained either with respect to $\gamma_E^{(m)}$ or to $\gamma_R^{(m)}$.
Notice that we can compute $\displaystyle \frac{\partial U_{m+1}(\xi)}{\partial \gamma_E^{(m)}}, \frac{\partial U_{m+1}(\xi)}{\partial \gamma_R^{(m)}}$ using \eqref{starting} with $k = m$ and
that $\displaystyle \frac{\partial U_{m}(\xi)}{\partial \gamma_*^{(m)}} = 0$, since polynomial $U_m$ does not depend on $\gamma_*^{(m)}$:
	\begin{eqnarray} 
	\label{WE}	W_{m+1}(\xi) &=&   U_{m}(\xi) \frac{\partial U_{m+1}(\xi)}{\partial \gamma_E^{(m)}} - U_{m+1}(\xi) \frac{\partial U_{m}(\xi)}{\partial \gamma_E^{(m)}} = (-1)^{m+1} U_m(\xi)^2 \\
	\label{WR}	W_{m+1}(\xi) &=& U_{m}(\xi) \frac{\partial U_{m+1}(\xi)}{\partial \gamma_R^{(m)}} - U_{m+1}(\xi) \frac{\partial U_{m}(\xi)}{\partial \gamma_R^{(m)}}= -U_{m-1}(\xi) U_m(\xi). 
	\end{eqnarray} Combining either \eqref{WE} or \eqref{WR} with \eqref{eqlemma3} we get the thesis.	}
\end{proof}
Now consider $a_{k+1}(\lambda)$, defined in \eqref{ak}, and compute it in one of the roots $\xi^{(k+1)}$ of $U_{k+1}$
\[
	0<a_{k+1}(\xi^{(k+1)})=U_k(\xi^{(k+1)})U_{k+1}^\prime(\xi^{(k+1)})
\]
we get that $U_k(\xi^{(k+1)})$ and $U_{(k+1)}^\prime(\xi^{(k+1)})$ must have the same sign. Now combining \eqref{eq10b} with (\ref{eq13}) and (\ref{eq14}) we obtain, at a root $\xi^{(k+1)}$,
\begin{eqnarray}
	\nonumber \textnormal{sgn}\left(\frac{\partial\lambda}{\partial\gamma_E^{(m)}}(\xi^{(k+1)})\right) &= &
	    -\textnormal{sgn}\left(\frac{\partial U_{k+1}}{\partial\gamma_E^{(m)}}(\xi^{(k+1)}) \right) 
	    \textnormal{sgn}(U'_{k+1} (\xi^{(k+1)}))\\ 
	    \nonumber &= & (-1)^m \textnormal{sgn}(U_k(\xi^{(k+1)})) \textnormal{sgn}\left(U'_{k+1}(\xi^{(k+1)})\right)
	    \\ \label{E} & =& (-1)^m \\
	\nonumber \textnormal{sgn}\left(\frac{\partial\lambda}{\partial\gamma_R^{(m)}}(\xi^{(k+1)})\right) &= &
	    -\textnormal{sgn}\left(\frac{\partial U_{k+1}}{\partial\gamma_R^{(m)}}(\xi^{(k+1)}) \right)  
	    \textnormal{sgn}(U'_{k+1} (\xi^{(k+1)})) \\ 
\nonumber 		    &= & 
	    \textnormal{sgn}\left(U_m(\xi^{(k+1)}) U_{m-1}(\xi^{(k+1)})\right) \textnormal{sgn}({U_k(\xi^{(k+1)})})
	    \textnormal{sgn}\left(U'_{k+1}(\xi^{(k+1)})\right) \\
        \label{R}
	    & = & \textnormal{sgn}\left(U_m(\xi^{(k+1)}) U_{m-1}(\xi^{(k+1)})\right).
    \end{eqnarray}
The sign of derivatives of $\lambda$ with respect to $\boldsymbol{\gamma}_R^{(m)}$
needs further  inspecting the sign of $U_m(\xi^{(k+1)}) U_{m-1}(\xi^{(k+1)})$.
	\begin{lemma}
	\label{lemmaCs}
	Let $\xi^{(k+1)}$ a root of $U_{k+1}(\lambda)$ and $C_s= U_s(\xi^{(k+1)}) U_{s-1}(\xi^{(k+1)})$ then 
	\begin{equation}
\label{bbb}			\textnormal{sgn}(C_s)
	\begin{cases}
		<  0 & \text{if} \qquad   \xi^{(k+1)}  = \xi_{1}^{(k+1)} \\
		>  0 & \text{if} \qquad   \xi^{(k+1)}  = \xi_{k+1}^{(k+1)}. \\
	\end{cases}
	\end{equation}
\end{lemma}
\begin{proof}
	To see the first inequality in \eqref{bbb} we observe that, since $\xi^{k+1}_{k+1}>\xi_j^{k+1}$ for all $j\leq k$, by Proposition \ref{prop1}, $U_s(\xi^{k+1}_{UB,+})$, $U_{s-1}(\xi^{k+1}_{UB,+})$ 
	have the same sign. The second inequality is obtained by observing that $\xi^{k+1}_{1} < \xi^{k+1}_j$ for all $j \leq k$; Proposition \ref{prop1} also
	implies that $U_s(\xi^{k+1}_{LB,-})$ and $U_{s-1}(\xi^{k+1}_{LB,-})$ have opposite signs. 

\end{proof}
In view of obtaining upper bounds for the negative and lower bounds for the positive roots, it is not possible to predict
the signs of $U_m(\xi^{(k+1)}) U_{m-1}(\xi^{(k+1)})$ in all cases, and hence the monotonicity of such roots
for a given polynomial degree $k+1$. However, we recall that our goal is to find bounds for the eigenvalues
of the preconditioned matrix, and therefore to compute the endpoints of the intervals in \eqref{N+1odd}.
To this aim, denote with $j$ the index such that $\xi^{(j)}  = \max\{\xi_{-,UB}^{(i)}, \ i = 2,4, \ldots, i \le N+1\}$.
In such a case, for any $m < j$,  polynomials $U_{m}(\xi^{(j)})$ and $U_{m-1}(\xi^{(j})$ must have the sign of $U_m(0)$ and $U_{m-1}(0)$, respectively.
In view of Proposition \ref{prop2} we have that 
\[ \textnormal{sgn}(U_{m}(0) U_{m-1}(0)) = 
\textnormal{sgn}(U_{m}(\xi^{(j)}) U_{m-1}(\xi^{(j})) = (-1)^m.\]
The same applies considering a polynomial $U_j(\lambda)$ in which  $j$ is such that $\xi^{(j)}  = \min\{\xi_{+,LB}^{(i)}, \ i = 1,3, \ldots, i \le N+1\}$.

\section{Bounds for the eigenvalues of the preconditioned matrix}
\label{sec5}
Summarizing the results of this Section, we are now ready to describe a procedure to compute the intervals \eqref{N+1odd} bounding
the eigenvalues of $\mathcal P^{-1} \mathcal A$ having $N+1$ blocks:
\begin{itemize}
	\item \textbf{Lower bound for the negative eigenvalues}. This is given by the smallest zero of the polynomial $U_{N+1}(\lambda)$.
		Regarding the choice of parameters, this is determined by the sign of the derivatives
                of the root with respect to both, and because we are minimizing the root. 
		In particular, for the $\gamma_E$ parameters, the sign of the derivative
		is $(-1)^m$, which  provides the values $\alpha_E^{(0)}, \beta_E^{(1)}, \alpha_E^{(2)}, \ldots $. By Lemma \ref{lemmaCs}, the derivative
		of the root with respect to the $\gamma_R$ parameters is negative, which, considering that we are seeking a minimum,
		provides the values
		$\beta_R^{(1)}, \beta_R^{(2)}, \ldots, \beta_R^{(N)}$.
	\item \textbf{Upper bounds for the negative eigenvalues}. We should consider evaluating the largest negative eigenvalue of 
		all polynomials $U_{2i}, 1 \le 2i \le N+1$.  The choice of the parameters is driven by the sign of the derivatives
		of the root with respect to both. The sign of the derivative of the root with respect to $\gamma_E^{(m)}$ and to
		$\gamma_R^{(m)}$  is $(-1)^m$. This, combined with the fact that we need to maximize the zero, provides the correct choice 
		of the parameters in this case, which are: 
	$\beta_E^{(0)}, \alpha_E^{(1)}, \beta_E^{(2)}, \ldots;\qquad  \alpha_R^{(1)}, \beta_R^{(2)}, \alpha_R^{(3)}, \ldots$.
		Call this root $\xi_-^{(2i)}$. Finally,  set \[b_{N+1} = \max\{\xi_-^{(2i)}, \ 1 \le 2i \le N+1\}. \]
	\item \textbf{Lower bounds for the positive eigenvalues}. We should consider evaluating the smallest positive eigenvalue of 
		all polynomials $U_{2i+1}, 0 \le 2i \le N$. As before, the sign of the derivative of the root with respect to $\gamma_E^{(m)}$ and to
                $\gamma_R^{(m)}$  is $(-1)^m$. This, combined with the fact that we need to minimize the zero, provides the correct choice
                of the parameters in this case, which are:
	$\alpha_E^{(0)}, \beta_E^{(1)}, \alpha_E^{(2)}, \ldots;\qquad  \beta_R^{(1)}, \alpha_R^{(2)}, \beta_R^{(3)}, \ldots$.
		Call this root $\xi_+^{(2i+1)}$. Finally,  set \[a_{N+1} = \min\{\xi_+^{(2i+1)}, \ 0 \le 2i \le N\}. \]
	\item \textbf{Upper bound for the positive eigenvalues}. This is given by the largest root of the polynomial $U_{N+1}(\lambda)$.
                For the $\gamma_E$ parameters, the sign of the derivative
		is $(-1)^m$, which provides the values $\beta_E^{(0)}, \alpha_E^{(1)}, \beta_E^{(2)}, \ldots $. By Lemma \ref{lemmaCs}, 
		the derivative
		of the root with respect to the $\gamma_R$ parameters is positive, which, considering that we are seeking a maximum,
		provides the values
		$\beta_R^{(1)}, \beta_R^{(2)}, \ldots, \beta_R^{(N)}$.
\end{itemize}
Summarizing, even if the possible combinations of the extremal values of the parameters are $2^{2N+1}$, we have obtained a \textit{linear} algorithm that requires the computation of the zeros of polynomials of degree $2, \ldots, N$ for one only combination of the parameters,
and four times the computation of the zeros of $U_{N+1}$, with as many different combinations of parameters.

\begin{Remark}
Regarding the double saddle point linear system, we report the bounds following the developments in Section 4, which coincide with those already proved in \cite{BMPP_COAP24}. 
	Consider the block matrix $\mathcal{A}$ with $N = 2$ (three blocks). Any eigenvalue $\lambda$ of ${\mathcal{P}}^{-1}\mathcal{A}$ satisfies
	\begin{equation}
		\label{boundsN2}
		\lambda       \in [\mu_{-}^{LB},\ \mu_{-}^{UB}] \cup [\mu_{+}^{LB},\ \mu_{+}^{UB}]
	\end{equation}
\textit{where}
	\begin{eqnarray}
		\mu_{-}^{LB} &=&\xi_{LB,-}^{(3)}(\alpha^{(0)}_E,\beta^{(1)}_E,\beta^{(1)}_R,\alpha^{(2)}_E,\beta^{(2)}_R)\nonumber, \\
		\mu_{-}^{UB} & =&\xi_{UB,-}^{(2)} (\beta_E^{(0)}, \alpha_E^{(1)}, \alpha_R^{(1)}) \nonumber, \\
		\label{UB3}	\mu_{+}^{LB} &=&\min \left\{\alpha^{(0)}_E,\xi_{LB,+}^{(3)}(\alpha^{(0)}_E,\beta^{(1)}_E,\beta^{(1)}_R,\alpha^{(2)}_E,\alpha^{(2)}_R)\right\}, \\
		\mu_{+}^{UB} & =&\xi_{UB,+}^{(3)} (\beta^{(0)}_E,\alpha^{(1)}_E,\beta^{(1)}_R,\beta^{(2)}_E,\beta^{(2)}_R). \nonumber
\end{eqnarray}
\end{Remark}
\section{Numerical experiments: Randomly-generated matrices.}
\label{Sec6}
We now undertake numerical tests to validate the theoretical bounds of Theorem \ref{lemma5}. We determine the extremal eigenvalues of ${\mathcal{P}}^{-1}\mathcal{A}$ on randomly-generated linear systems. Specifically, we considered a simplified case with $E_i\equiv 0, i > 0$ and run 
several different test cases, combining the values of the extremal eigenvalues (Table \ref{tab:1}) of the symmetric positive definite matrices involved.
In particular

\begin{itemize}
	\item Case $N = 2$. We have three parameters with 6 different endpoints of the corresponding intervals. Overall, we run $3^6=729$ test cases.
	\item Case $N = 3$. We have 4 parameters with 4 different endpoints of the corresponding intervals. Overall, we run $4^4=256$ test cases.
	\item Case $N = 4$. We have 5 parameters with 4 different endpoints of the corresponding intervals. Overall, we run $5^4=1024$ test cases.
\end{itemize}

\begin{table}[h!]
    \centering
    \begin{tabular}{|l|ccc|c|ccc}
    \hline
	    $\alpha_E^{(0)},\ \alpha_R^{(1)},\ \alpha_R^{(2)}$ & 0.1 & 0.3 & 0.9 &
         $\alpha_E^{(0)},\ \alpha_R^{(1)}, \ldots \alpha_R^{(N)}$ & 0.1 & 0.9\\
	    $\beta_E^{(0)},\ \beta_R^{(1)},\ \beta_R^{(2)}$ & 1.2 & 1.8 & 5 &
         $\beta_E^{(0)},\ \beta_R^{(1)}, \ldots \beta_R^{(N)}$ & 1.2 & 5\\
	 \hline
	    \multicolumn{3}{c}{Case $N = 2$} && 
	    \multicolumn{3}{c}{Cases $N > 2$}  \\
    \end{tabular}
    \caption{Extremal eigenvalues of the relevant symmetric positive definite matrices used in the verification of the bounds.}
    \label{tab:1}
\end{table}
\begin{figure}[h!]
    \centering
    \includegraphics[width=0.9\linewidth]{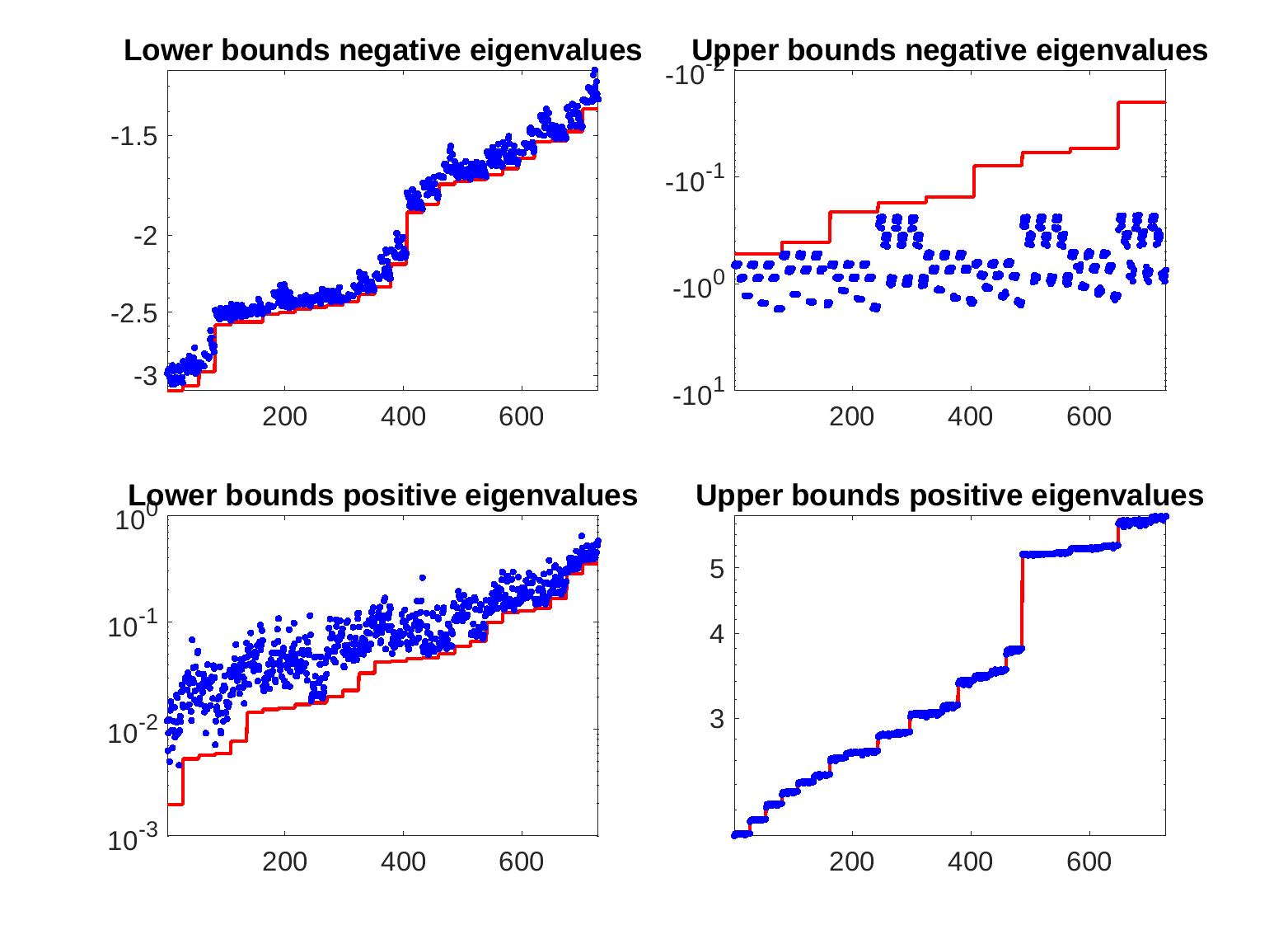}
    \caption{Double saddle-point linear system. Extremal eigenvalues of the preconditioned matrix (blue dots) and bounds (red line) after $10$ runs with each combination of the parameters from Table \ref{tab:1}.}
    \label{fig:2}
\end{figure}

Each test case has been run $10$ times, generating random matrices that satisfy the relevant spectral properties, and we record the most extreme eigenvalues for each test case.

In more detail, the dimensions $n_0$, $n_1, \ldots$ are computed using the closest integer to $50+10*\texttt{rand}$, using \texttt{MatLab}'s \texttt{rand} function, recomputing as necessary to ensure that $n_k \geq n_{k+1}, \forall k \le N-1$. 
Matrices $A_k$ and $B_k$ are computed using \texttt{MATLAB}'s \texttt{randn} function, whereupon we take the symmetric part of $A_k$ and then add an identity matrix times the absolute value of the smallest eigenvalue to ensure symmetric positive semi-definiteness (definiteness if $k=0$). We then choose $\widehat{S}_0$ as a linear combination of $A$ and the identity matrix, such that the eigenvalues of $E_0$ are contained in $[\alpha_E^{(0)}, \beta_E^{(0)}]$, and similarly to construct $\widehat{S}_k$ for any $k$. 

In Figures \ref{fig:2}, \ref{fig:3}, and \ref{fig:4}, we present the (ordered) computed extremal eigenvalues and theoretical bounds for the $N=2, 3$ and $N=4$ cases. 
We notice that the plots indicate (for these problems) that three out of four bounds capture the behavior of the eigenvalues very well, while only the upper bounds on the negative eigenvalues (for $N = 2,4$) and the lower
bounds on the positive eigenvalues ($N = 3$) are not as tight.
\begin{figure}[h!]
    \centering
    \includegraphics[width=0.9\linewidth]{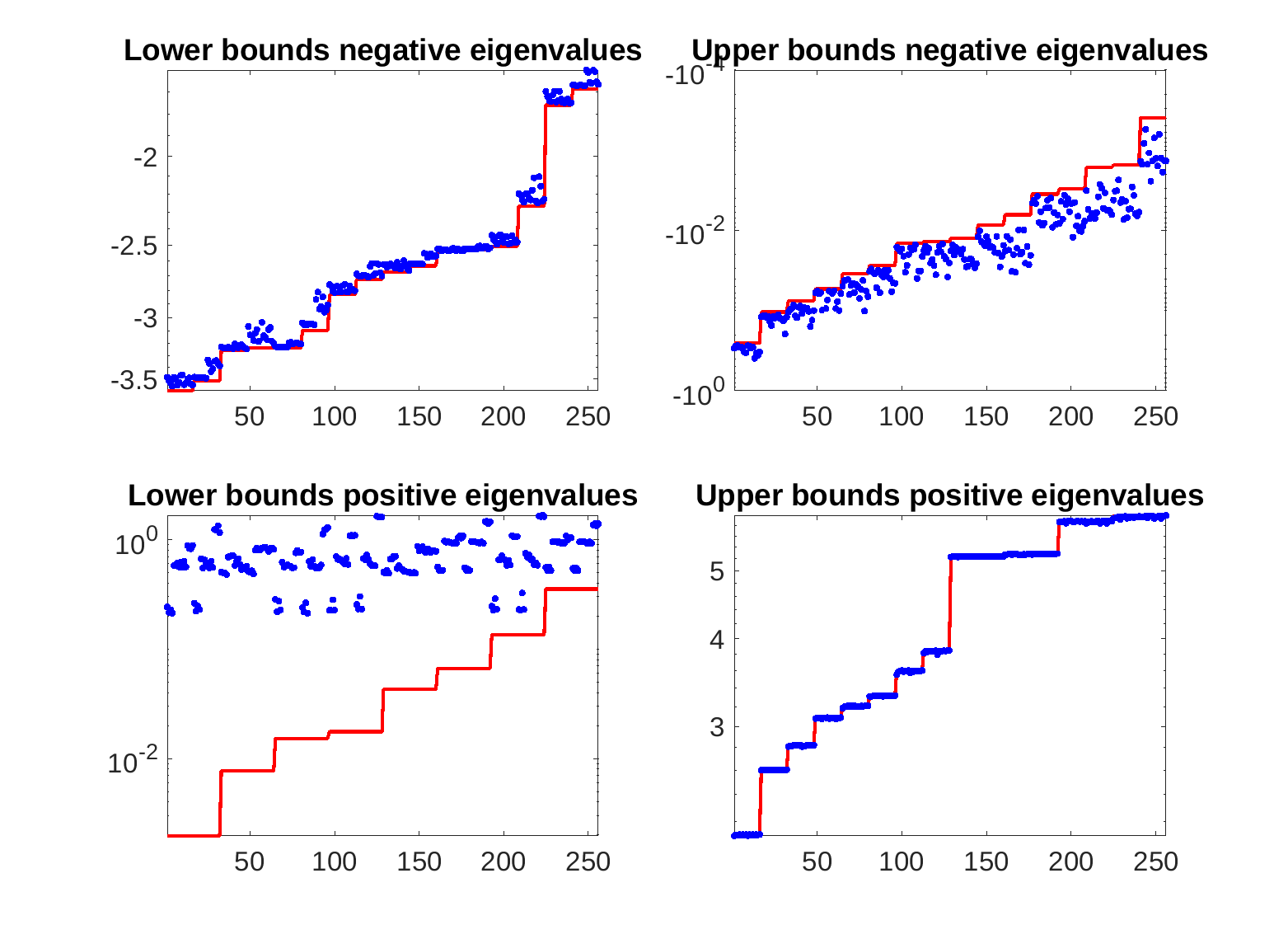}
	\vspace{-8mm}
    \caption{Triple saddle-point linear system. Extremal eigenvalues of the preconditioned matrix (blue dots) and bounds (red line) after $10$ runs with each combination of the parameters from Table \ref{tab:1}.}
    \label{fig:3}
\end{figure}
\begin{figure}[h!]
    \centering
    \includegraphics[width=0.9\linewidth]{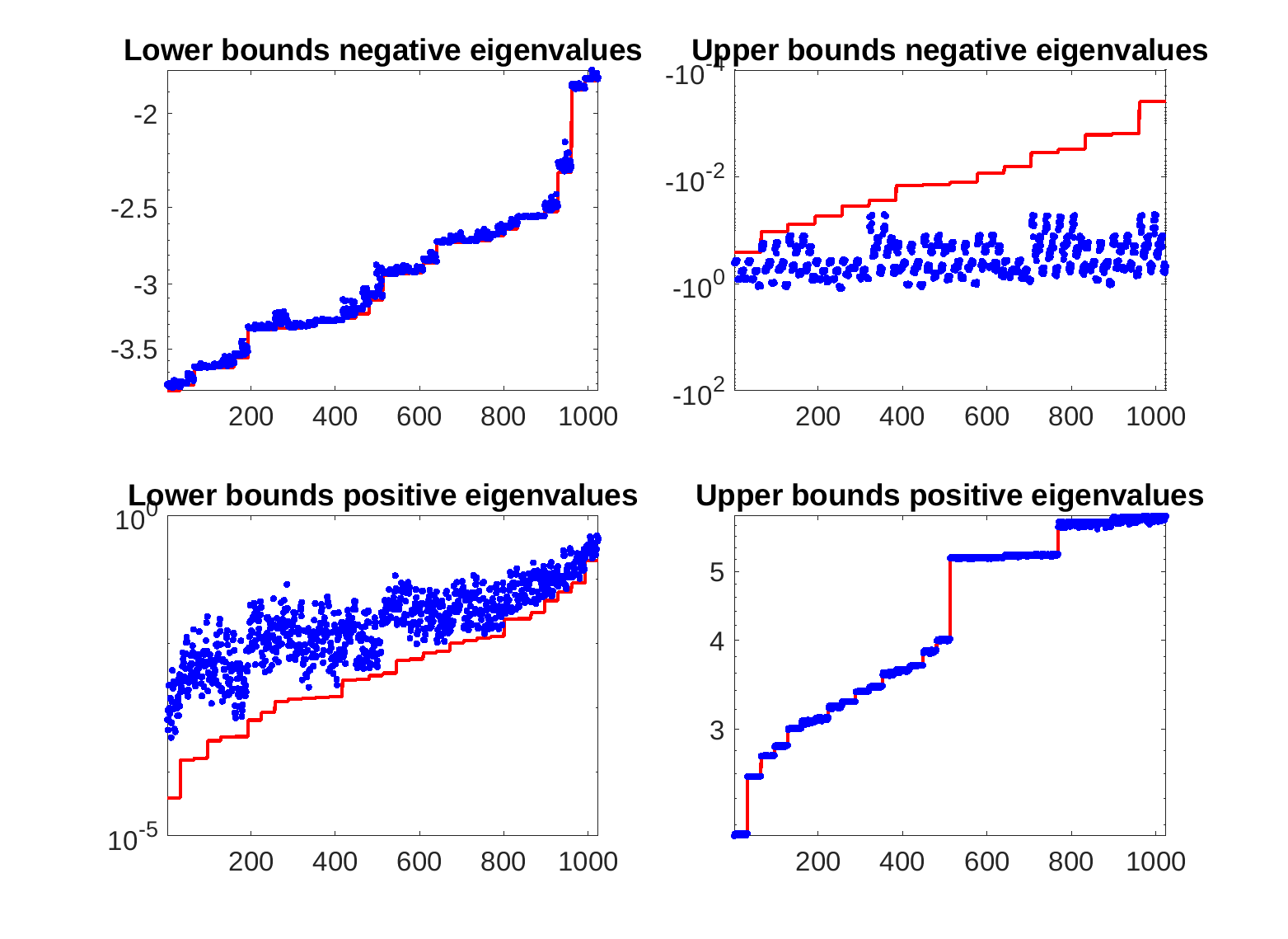}
    \vspace{-8mm}
    \caption{Quadruple saddle-point linear system. Extremal eigenvalues of the preconditioned matrix (blue dots) and bounds (red line) after $10$ runs with each combination of the parameters from Table \ref{tab:1}.}
    \label{fig:4}
\end{figure}

\textcolor{black}
{
\section{Numerical experiments: A realistic test case}
\label{Sec7}
A numerical experiment, concerning the mixed form of Biot's poroelasticity equations, is used to investigate the quality of the bounds for the eigenvalues of the preconditioned matrix and the effectiveness of the proposed triangular preconditioners on a realistic problem. 
For the PDEs and boundary conditions governing this model, we refer to the description in \cite{BerFerMar25}.
As a reference test case for the experiments, we considered the porous cantilever beam problem originally introduced in \cite{WheelPhil} and already used in \cite{Frigo202136}. The domain is the unit cube with no-flow boundary conditions along all sides, zero displacements along the left edge, and a uniform load applied on top.  The material properties are described in \cite[Table 5.1]{BerFerMar25}.
We considered the Mixed-Hybrid Finite Element discretization of the coupled Biot equations, which gives rise to a double saddle point linear system.
The test case refers to a unitary cubic domain, uniformly discretized using a $h = 0.05$ mesh size in all the spatial dimensions. The size of the block matrices and the overall number of nonzeros in the double saddle-point system are reported in Table \ref{tab:size}.
\begin{table}[h!]
	\caption{3D cantilever beam problem on a unit cube: size and number of nonzeros of the test matrices. }
	\label{tab:size}
\begin{center}
\begin{tabular}{r|rrrcr}
	$1/h$&$n_0$  & $n_1$  &$n_2$   &$n_0+n_1+n_2$  & nonzeros\\
	\hline            
	20 &27783  &8000  &25200 &60983  &2.4$\times 10^6$  \\
\end{tabular}
\end{center}
\end{table}
\subsection{Handling the case $n_2 > n_1$}
The theory previously developed is based on the assumption $n_0 \ge n_1 \ge n_2 \ldots$, which is not verified in this problem,
where $n_2 = 25200 > n_1 = 8000$. The main consequence of this is that $R_2 R_2^\top$ is singular and $\gamma_R^{(2)} = 0$.
This drawback can be circumvented by attacking the eigenvalue problem \eqref{eq1} in a different way. Let us write \eqref{eq1}
for a double saddle-point linear system:
\begin{equation}
    \begin{array}{ccccccccc}
        (E_0-\lambda I)u_1 & + & R_1^Tu_2 & = & 0\\
        R_1u_1 & + &-(E_1+\lambda I) u_2 & +& R_2^Tu_3 &  = & 0 \\
         & & R_2u_2 & + & (E_2-\lambda I)u_3 & = & \ 0.
    \end{array}
    \label{eq11}
\end{equation}
As in the proof of Theorem 1, we have, for all $\lambda \not \in [\alpha_E^{(0)}, \beta_E^{(0)}]$
\begin{equation} \label{u2} (R_1Y_1(\lambda)^{-1}R_1^T-\lambda I -E_1)u_2=R_2^Tu_3. \end{equation}
Assuming now that $\lambda \not \in [\alpha_E^{(2)}, \beta_E^{(2)}]$, we obtain $u_3$ from the third of \eqref{eq11}
\[ u_3 = -(E_2-\lambda I)^{-1} R_2u_2.\]
Substituting this into \eqref{u2} yields
\[ \left(R_1Y_1(\lambda)^{-1}R_1^T-\lambda I -E_1 + R_2^\top (E_2-\lambda I)^{-1} R_2)\right)u_2 = 0. \]
Premultiplying now by $u_2^\top$ and defining $w = R_1^\top u_2$ and $z = R_2 u_2$,  yields
\[ \frac{w^\top Y_1(\lambda)^{-1} w}{w^\top w} \gamma_R^{(1)}  -\lambda -\gamma_E^{(1)} + \frac{z^\top (E_2-\lambda I)^{-1} z}{z^\top z} \overline \gamma_R^{(2)} = 0. \]
Before proceeding we notice that $\overline \gamma_R^{(2)}$, differently from definition \eqref{gammaER}, is the Rayleigh quotient of $R_2^\top R_2$.
Due to the fact that $R_2$ has full column rank we can bound $\overline \gamma_R^{(2)}$ as $0 < \alpha_R^{(2)} \le \overline \gamma_R^{(2)} \le \beta_R^{(2)}.$ 
Hence, from now on, we will refer to $\overline \gamma_R^{(2)}$ as simply $\gamma_R^{(2)}$.
Applying Lemma \ref{lemma2} to $Y_1$ and Lemma \ref{lemma1} to $E_2-\lambda I$, and using the recurrence definition \eqref{eq2}, we rewrite the left-hand side of the previous as
\[ -\frac{U_2(\lambda)}{U_1(\lambda)} + \frac{1}{\gamma_E^{(2)} - \lambda} \gamma_R^{(2)}  =  \frac{U_3(\lambda)}{U_1(\lambda) (\lambda - \gamma_E^{(2)})},  \]
which tells that any eigenvalue of $\mathcal P^{-1} \mathcal A$ not lying in 
$[\alpha_E^{(0)} , \beta_E^{(0)}] \cup [\alpha_E^{(2)}, \beta_E^{(2)}]$ is a zero of $U_3(\lambda)$, or, equivalently, that
$\lambda \in \mathcal I_1 \cup \mathcal{I}_{E_2} \cup \mathcal I_3$. \\[.5em]
\noindent
To approximate the $(1,1)$ block, we employed the classical incomplete Cholesky factorization (IC) with fill-in based on a drop tolerance $\delta$.
The approximations $\widehat S_1$ and $\widehat S_2$ of the Schur complements are as in \cite{BerFerMar25}, which prove completely
scalable with the problem size. Even if the IC preconditioner does not scale with the discretization parameter, it is helpful to conduct
a spectral analysis as a function of the drop tolerance. \\
\noindent
In Table \ref{Tab:indicators} we report the values of the indicators for three runs corresponding to the values of the threshold parameter $\delta = 10^{-3}, 10^{-4}, 10^{-5}$, and $10^{-6}$. In combination with these parameters, the bounds obtained by the theory are reported in Table \ref{Bounds_and_its}, where we also provide the number of MINRES iterations
to solve $\mathcal A x = b$, for the three instances described. The right-hand side has been selected by imposing a true solution of all ones and the iteration
is stopped as soon as the relative residual norm is below $10^{-14}$. The CPU times (in seconds) refer to a Matlab 
implementation on an Intel Core Ultra 7 165U Notebook at 3.8 GHz with 32 GB RAM.
%
\begin{table}[h!]
\begin{center}
	\begin{tabular}{l|lrlrlrlr}
    & \multicolumn{2}{c}{$\delta = 10^{-3}$} 
    & \multicolumn{2}{c}{$\delta = 10^{-4}$} 
    & \multicolumn{2}{c}{$\delta = 10^{-5}$} 
    & \multicolumn{2}{c}{$\delta = 10^{-6}$} \\
		 \hline
         $I_{E_0}$ & $[0.0099 $, & $ 1.2358]$ & $[0.0422 $, & $ 1.2036]$ & $[0.2563 $, & $ 1.2232]$& 
		 $[0.9600 $, & $ 1.0118]$ \\
		 $I_{E_1} $ & $[5\times 10^{-4} $, & $ 0.0889]$ \\
         		 $I_{E_2}$ & $[0.9976            $,  & $ 1.0001]$ \\
               $I_{R_1}$ & $[2\times 10^{-4} $, & $ 0.7790]$  \\
		 $I_{R_2}$ & $[3\times 10^{-5} $, & $ 0.0023]$   \\
	\end{tabular}
\end{center}
\caption{Extremal values of the indicators for different values of the threshold parameter $\delta$.}
\label{Tab:indicators}
\end{table}
\begin{table}[h!]
\color{black}       \begin{center}
\begin{tabular}{cl|rrrr|rr}
& &$\mu_-^{\text{LB}}$ & $\mu_-^{\text{UB}}$ & $\mu_+^{\text{LB}}$ & $\mu_+^{\text{UB}}$ & $n_{it}$&CPU\\
\hline
$\delta = 10^{-3}$ &Bounds &-1.4308  & -0.0012    &0.0101  &  1.6958 & 248  & 1.63\\
& Exact Eigvs & -1.0573 & -0.3113 & 0.0108 &  1.4741 \\ \hline
 $\delta = 10^{-4}$ &Bounds&-1.4221  & -0.0012    &0.0424  &  1.6706 & 132  & 1.49\\ 
& Exact Eigvs &-1.0573 & -0.3113 & 0.0455 &  1.4787 
\\ \hline
$\delta = 10^{-5}$ &Bounds&  -1.3693 &  -0.0012  &  0.2565  &  1.6859  & 83 & 2.06\\
& Exact Eigvs &-1.0573 & -0.3113 & 0.2691 &  1.4856 
\\ \hline
$\delta = 10^{-6}$ &Bounds&  -1.2436  & -0.0012 &   0.9601   & 1.5241 & 64 & 2.90 \\             
& Exact Eigvs &-1.0573 & -0.3113 & 0.9629 &  1.4843 
\\
\end{tabular} 
\end{center}
 \caption{3D cantilever beam problem on a unit cube: eigenvalue bounds vs real eigenvalues. The number of MINRES iterations $n_{it}$ and the CPU time (in seconds) of the iterative procedure, with different approximations of the $(1,1)$ block $A$, are also provided.
        }
        \label{Bounds_and_its}
\end{table}
}

\textcolor{black}
{
From Table \ref{Bounds_and_its} we see that: (i) The accuracy on the approximation of the $(1,1)$ block
influences mostly the smallest positive eigenvalue, which is well captured by our bounds; (ii) all in all the bounds
correctly describe the extremal eigenvalues, with the only exception of the largest negative eigenvalue, for which the bound is looser, as already observed in Section 6. 
}
\section{Conclusions}
\label{Sec8}
We analyzed spectral bounds for the application of block diagonal approximated Schur complement preconditioners to symmetric multiple saddle-point systems. The analysis is based on recursively defined polynomials, the zeros of which coincide with the eigenvalues of interest. We established general bounds for the spectrum in the case of approximated preconditioners. The results obtained demonstrate the consistency of the theoretical predictions with numerically computed eigenvalues, on synthetic test cases, as well as in the solution of a realistic problem arising from the discretization of a PDE modeling a 3D poroelasticity equation.
%
\subsection*{Acknowledgements}
The work of AM was carried out within the PNRR research activities of the consortium iNEST (Interconnected North-Est Innovation Ecosystem) funded by the European Union Next-GenerationEU
(PNRR -- Missione 4 Componente 2, Investimento 1.5 -- D.D. 1058 23/06/2022, ECS$\_$00000043). LB acknowledges financial support under the PNRR research activity, Mission 4, Component 2, Investment 1.1,
funded by the EU Next-GenerationEU -- \#2022AKNSE4\_005 (PE1) CUP C53D2300242000.
This manuscript reflects only the Authors' views and opinions, neither the European Union nor the European Commission can be considered responsible for them.   LM and AM are members of the INdAM research group GNCS.

\end{document}